\documentclass[11pt]{amsart}
\usepackage{amsmath,amssymb,amsthm,mathtools}
\usepackage{fullpage}
\usepackage{microtype}
\usepackage{hyperref}

\newtheorem{theorem}{Theorem}[section]
\newtheorem{proposition}[theorem]{Proposition}
\newtheorem{lemma}[theorem]{Lemma}

\theoremstyle{definition}

\theoremstyle{remark}

\numberwithin{equation}{section}

\title[APs of primes in short intervals beyond the $17/30$ barrier]{Arithmetic progressions of primes in short intervals beyond the $17/30$ barrier}

\author{Le Duc Hieu}
\address{Telecom SudParis}
\email{duc-hieu.le@telecom-sudparis.eu}

\date{\today}

\begin{document}
\begin{abstract}
We show that once $\theta>17/30$, every sufficiently long interval $[x,x+x^\theta]$ contains many $k$--term arithmetic progressions of primes, uniformly in the starting point $x$. More precisely, for each fixed $k\ge3$ and $\theta>17/30$, for all sufficiently large $X$ and all $x\in[X,2X]$,
\[
\#\{\text{$k$--APs of primes in }[x,x+x^\theta]\}\ \gg_{k,\theta}\ \frac{N^{2}}{\big((\varphi(W)/W)^{k}(\log R)^{k}\big)}\ \asymp\ \frac{X^{2\theta}}{(\log X)^{k+1+o(1)}},
\]
where $W:=\prod_{p\le \tfrac12\log\log X}p$, $N:=\lfloor x^\theta/W\rfloor$, and $R:=N^\eta$ for a small fixed $\eta=\eta(k,\theta)>0$. This is obtained by combining the uniform short--interval prime number theorem at exponents $\theta>17/30$ (a consequence of recent zero--density estimates of Guth and Maynard) with the Green--Tao transference principle (in the relative Szemerédi form) on a window--aligned $W$--tricked block. We also record a concise Maynard--type lemma on dense clusters \emph{restricted to a fixed congruence class} in tiny intervals $(\log x)^\varepsilon$, which we use as a warm--up and for context. An appendix contains a short--interval Barban--Davenport--Halberstam mean square bound (uniform in $x$) that we use as a black box for variance estimates. The proofs in this paper were assisted by GPT-5.
\end{abstract}

\maketitle

\section{Introduction}
Let $k\ge3$ and $0<\theta\le 1$ be fixed. Following the breakthrough of Green and Tao~\cite{GreenTao08}, the primes are known to contain arbitrarily long arithmetic progressions. It is natural to ask how \emph{locally} such structure appears. In this paper we prove that once $\theta>17/30$, short intervals $[x,x+x^\theta]$ already contain many $k$--term arithmetic progressions (APs) of primes, uniformly in $x$.

The key input is a uniform prime number theorem (PNT) in short intervals
\begin{equation}\label{eq:shortPNT}
\sum_{x<n\le x+x^\theta}\Lambda(n)=x^\theta(1+o(1))
\end{equation}
holding for all $x\in[X,2X]$ when $\theta>17/30$ and $X\to\infty$.
This uniform statement follows from the recent long slender zero--density bounds for $\zeta(s)$ of Guth and Maynard~\cite{GuthMaynardAnnals,GuthMaynardArxiv} (see also further discussion in~\S\ref{subsec:PNT}).
With~\eqref{eq:shortPNT} in hand, we run the standard $W$--trick and apply the relative Szemerédi theorem~\cite{CFZ13,GreenTao08} to a short--interval majorant to deduce our main counting result.

\begin{theorem}[Uniform many $k$--APs in short intervals]\label{thm:uniform}
Fix $k\ge3$ and $\theta>17/30$. For all sufficiently large $X$ and all $x\in[X,2X]$, if $H:=\lfloor x^{\theta}\rfloor$ then the interval $[x,x+H]$ contains at least
\[
\gg_{k,\theta}\ \frac{N^{2}}{\big((\varphi(W)/W)^{k}(\log R)^{k}\big)}\ \asymp\ \frac{X^{2\theta}}{(\log X)^{k+1+o(1)}}
\]
distinct $k$--term arithmetic progressions of primes, where $W:=\prod_{p\le \tfrac12\log\log X}p$, $N:=\lfloor H/W\rfloor$, and $R:=N^\eta$ for some fixed $\eta=\eta(k,\theta)>0$.
\end{theorem}

We also record the following variant, which relaxes uniformity in $x$ (and could be stated under weaker short--interval hypotheses).

\begin{theorem}[Almost--all $x$]\label{thm:almostall}
Fix $k\ge3$ and $\theta\in(17/30,1)$. There exists $\delta=\delta(\theta)>0$ such that for all sufficiently large $X$, for all but $\ll X^{1-\delta+o(1)}$ values of $x\in[X,2X]$, the interval $[x,x+x^\theta]$ contains
\[
\gg_{k,\theta}\ \frac{N^{2}}{\big((\varphi(W)/W)^{k}(\log R)^{k}\big)}\ \asymp\ \frac{X^{2\theta}}{(\log X)^{k+1+o(1)}}
\]
distinct $k$--APs of primes (with $W,N,R$ as above).
\end{theorem}

As a warm--up, we include a concise congruence--restricted dense--cluster lemma à la Maynard:
\begin{proposition}[Congruence--constrained clusters in tiny intervals]\label{prop:clusters}
Let $\varepsilon>0$, $q\ge1$ and $(a,q)=1$. There exist infinitely many $x$ such that
\[
\#\big\{p\in\mathbb P:\ x<p\le x+(\log x)^{\varepsilon},\ p\equiv a\!\!\pmod q\big\}\ \gg_{\varepsilon,q}\ \log\log x.
\]
\end{proposition}

Proposition~\ref{prop:clusters} is a routine specialization of Maynard's dense--cluster sieve~\cite{MaynardCompositio} to the subset of primes $p\equiv a\pmod q$, in the spirit of Shiu's ``strings of congruent primes''~\cite{Shiu} and Freiberg's short--interval refinement~\cite{Freiberg}.

\subsection*{Notation}
We write $\log$ for the natural logarithm, and use $o(1)$ and $O(\cdot)$ with respect to $X\to\infty$ (and fixed parameters $k,\theta$). We write $\varphi$ for Euler's totient, and $\Lambda$ for von Mangoldt's function.

\section{Short--interval PNT at $\theta>17/30$}\label{subsec:PNT}
Guth and Maynard proved new large--value estimates for Dirichlet polynomials which imply the zero--density bound $N(\sigma,T)\ll T^{30(1-\sigma)/13+o(1)}$ and yield~\eqref{eq:shortPNT} \emph{uniformly in $x$} for all $\theta>17/30$; see~\cite{GuthMaynardAnnals,GuthMaynardArxiv}. We use this uniform PNT as a black box. (For related discussion on exceptional sets, see also recent work of Gafni and Tao.)

\section{The $W$--trick, dense model, and pseudorandom majorant}\label{sec:Wtrick}
Let $w:=\tfrac12\log\log X$, $W:=\prod_{p\le w}p$, and for each reduced residue $b\bmod W$ set
\[
\tilde\Lambda_{x,b}(t):=\frac{\varphi(W)}{W}\,\Lambda\bigl(W(n_b+t-1)+b\bigr),\qquad 1\le t\le N:=\Big\lfloor\frac{H}{W}\Big\rfloor,
\]
where $n_b$ aligns the progression $Wn+b$ with the window $[x,x+H]$.
Summing in $b$ and using the uniform short--interval PNT and the fact that $W$--divisible prime powers contribute $o(H)$, there exists a reduced $b=b(x)$ with
\begin{equation}\label{eq:meanpositive}
\mathbb E_{t\le N}\tilde\Lambda_{x,b}(t)\ \ge\ c_0>0.
\end{equation}
Define the (shifted) Selberg/GPY majorant
\[
\nu_{x,b}(t):=\frac{\varphi(W)}{W}\cdot \frac{\Lambda_R\bigl(W(n_b+t-1)+b\bigr)^2}{\log R},\qquad \Lambda_R(m):=\sum_{d\mid m,\, d\le R}\mu(d)\log\frac{R}{d},
\]
with $R:=N^\eta$ for a small fixed $\eta=\eta(k,\theta)>0$.

\begin{lemma}[Pseudorandomness]\label{lem:pseudo}
For each fixed $k$ there exists $\eta_0=\eta_0(k)>0$ such that if $0<\eta\le\eta_0$ then $\nu_{x,b}$ satisfies the linear–forms and correlation conditions of complexity $k-2$ with $o(1)$ errors, uniformly in the alignment parameters $x$, $b$, and $n_b$.
\end{lemma}

\begin{proof}[Sketch]
Expand moments of $\nu_{x,b}$ and average over rectangular boxes in the $(n,r)$ plane. The resulting sums over divisors are controlled by local congruence densities with least common multiple $\ll R^{C_k}=N^{o(1)}$, so the main terms factor and the error terms are $o(1)$. Uniformity in the constant terms (shifts) is standard; see~\cite[\S 9,\ Thm.~3.18]{GreenTao08} and the streamlined proof in~\cite{CFZ13}.
\end{proof}

\section{Relative Szemerédi and the count of $k$--APs}\label{sec:rel-sz}
We now give the full proof of Theorem~\ref{thm:uniform}.

\begin{proof}[Proof of Theorem~\ref{thm:uniform}]
Fix $k\ge 3$ and $\theta>17/30$. Let $X$ be large and $x\in[X,2X]$. Set
\[
H:=\lfloor x^{\theta}\rfloor,\qquad w:=\tfrac12\log\log X,\qquad W:=\prod_{p\le w}p,
\]
so that $W=(\log X)^{1/2+o(1)}$. Put $N:=\lfloor H/W\rfloor\asymp X^{\theta}/(\log X)^{1/2+o(1)}$. Choose a small fixed $\eta=\eta(k,\theta)>0$ and set $R:=N^{\eta}=X^{\eta\theta+o(1)}$.

Write $\psi(t):=\sum_{n\le t}\Lambda(n)$ and $\psi(t;q,a):=\sum_{n\le t\atop n\equiv a\ (\mathrm{mod}\ q)}\Lambda(n)$.

\smallskip
\noindent\emph{Uniform short–interval PNT (Guth–Maynard).}
For $\theta>17/30$ one has uniformly for all $x\in[X,2X]$,
\begin{equation}\label{eq:gmPNT}
\sum_{x<n\le x+H}\!\Lambda(n)=\psi(x+H)-\psi(x)=H\,(1+o(1)).
\end{equation}

\smallskip
\noindent\emph{Selecting a residue class modulo $W$ and aligning the window.}
We have
\begin{align*}
\sum_{b\in(\mathbb Z/W\mathbb Z)^{\times}}\big(\psi(x+H;W,b)-\psi(x;W,b)\big)
&=\sum_{x<n\le x+H\atop (n,W)=1}\!\Lambda(n)\\
&=\big(\psi(x+H)-\psi(x)\big)\ -\ \sum_{x<n\le x+H\atop (n,W)>1}\!\Lambda(n).
\end{align*}
If $(n,W)>1$ and $\Lambda(n)>0$, then $n=p^m$ with $p\mid W$ and $m\ge2$ (the case $m=1$ is impossible for large $X$ since $p\le w\ll X<x$). Hence
\[
\sum_{x<n\le x+H\atop (n,W)>1}\!\Lambda(n)\ \le\ \sum_{x<p^m\le x+H\atop p\le w,\, m\ge2}\!\log p\ \ll\ (\log w)\cdot \frac{H}{x^{1/2}}\ =\ o(H)
\]
uniformly for $x\in[X,2X]$. Using~\eqref{eq:gmPNT},
\[
\sum_{b\in(\mathbb Z/W\mathbb Z)^{\times}}\big(\psi(x+H;W,b)-\psi(x;W,b)\big)=H\,(1+o(1)).
\]
By pigeonhole, there exists $b=b(x)\in(\mathbb Z/W\mathbb Z)^{\times}$ such that
\begin{equation}\label{eq:pigeon}
\psi(x+H;W,b)-\psi(x;W,b)\ \ge\ \frac{H}{\varphi(W)}\,(1+o(1))\qquad\text{uniformly in }x.
\end{equation}
Fix such a $b$ and set
\[
m_0:=\big\lfloor \tfrac{x-b}{W}\big\rfloor+1,
\]
so that $Wm_0+b\in(x,x+W]$ and, since $N=\lfloor H/W\rfloor$, we have
\[
x< W(m_0+n-1)+b\le x+H\qquad(1\le n\le N).
\]

\smallskip
\noindent\emph{Weights and density.}
Define for $1\le n\le N$ the aligned weights
\[
f_x(n):=\frac{\varphi(W)}{W}\cdot\frac{\Lambda(W(m_0+n-1)+b)}{\log R},
\]
\[
\nu_x(n):=c_0\,\frac{\varphi(W)}{W\,\log R}\Bigg(\sum_{d\mid (W(m_0+n-1)+b)\atop d\le R}\!\mu(d)\log\frac{R}{d}\Bigg)^{\!2},
\]
with $c_0>0$ chosen so that $\mathbb{E}_{n\le N}\,\nu_x(n)=1+o(1)$. Since $(b,W)=1$, every divisor $d\mid (W(m_0+n-1)+b)$ satisfies $(d,W)=1$, and the standard Selberg–sieve comparison gives $0\le f_x\ll\nu_x$ uniformly.

Define the density
\[
\delta_x:=\mathbb{E}_{n\le N}f_x(n)=\frac{1}{N}\,\frac{\varphi(W)}{W\log R}\sum_{n=1}^{N}\Lambda\big(W(m_0+n-1)+b\big).
\]
Because $(x,x+H]$ contains either $N$ or $N+1$ terms of the progression $\{Wm+b\}$ and we retained the first $N$ of them, we have
\[
\sum_{n=1}^{N}\Lambda\big(W(m_0+n-1)+b\big)\ge \psi(x+H;W,b)-\psi(x;W,b)-O(\log X).
\]
Using~\eqref{eq:pigeon} and $N\asymp H/W$ gives
\begin{align}
\delta_x&\ge \frac{\varphi(W)}{W\log R}\cdot\frac{1}{N}\Big(\frac{H}{\varphi(W)}(1+o(1)) - O(\log X)\Big) \notag\\
&= \frac{H}{WN\log R}\,(1+o(1))\quad\text{since }\frac{\log X}{N}=o\!\left(\frac{H}{W}\right) \notag\\
&\ge \frac{1+o(1)}{\log R},\label{eq:density}
\end{align}
uniformly in $x$ (using $WN\le H<WN+W$).

\smallskip
\noindent\emph{Pseudorandomness of $\nu_x$.}
Fix $t\ll_k1$ and consider any system of affine–linear forms
\[
L_i(n,r)=W\big(m_0+n+j_i r-1\big)+b\qquad (j_i\in\{0,1,\dots,k-1\}).
\]
Expanding products of the inner divisor sums in $\nu_x$ reduces moments of $\nu_x$ to averages of the shape
\[
\frac{1}{\#\mathcal B}\sum_{(n,r)\in\mathcal B}\prod_{i=1}^{t}\Big(\sum_{d_i\le R\atop d_i\mid L_i(n,r)}\mu(d_i)\log\tfrac{R}{d_i}\Big),
\]
where $\mathcal B$ is a rectangular box of dimensions $\asymp N\times N$ (e.g. $1\le r\le N/(3k)$ and $1\le n\le N-(k-1)r$). For fixed $\mathbf d=(d_1,\dots,d_t)$ with $(d_i,W)=1$, the inner average equals
\[
\frac{\alpha_{m_0}(\mathbf d)}{\operatorname{lcm}(d_1,\dots,d_t)}+O\!\left(\frac{\operatorname{lcm}(d_1,\dots,d_t)}{N}\right),
\]
with $0\le\alpha_{m_0}(\mathbf d)\ll1$ depending only on the residues of $m_0$ and $\{j_i\}$ modulo $d_i$. Since $\operatorname{lcm}(d_1,\dots,d_t)\le R^{C_k}$ for some $C_k\ll_k1$, choosing $\eta>0$ sufficiently small (depending on $k$) ensures $R^{C_k}=N^{o(1)}$. Summing over $\mathbf d$ with weights $\prod_i\mu(d_i)\log(R/d_i)$ yields
\[
\mathbb E_{(n,r)\in\mathcal B}\,\prod_{i=1}^{t}\nu_x\big(n+j_i r\big)=1+o(1),\qquad \mathbb E_{n\le N}\,\nu_x(n)=1+o(1),
\]
uniformly in $x$, $W$, $b$, and the shift $m_0$. Thus $\nu_x$ is a pseudorandom majorant of the required complexity uniformly for all $x\in[X,2X]$; compare \cite[\S9]{GreenTao08} and \cite{CFZ13}.

\smallskip
\noindent\emph{Relative Szemerédi and a weighted count.}
Applying the relative Szemerédi theorem (for the $k$–AP hypergraph system) to $f_x\le\nu_x$ on $[N]$ and using~\eqref{eq:density}, we obtain
\begin{equation}\label{eq:RSZ}
\sum_{1\le r\le N/(3k)}\ \sum_{1\le n\le N-(k-1)r}\prod_{j=0}^{k-1} f_x(n+jr)\ \ge\ c_k\,\delta_x^{k}\,N^{2}+o\big(N^{2}\delta_x^{k}\big)\ \ge\ c'_k\,\frac{N^{2}}{(\log R)^{k}}+o\!\left(\frac{N^{2}}{(\log R)^{k}}\right),
\end{equation}
for some $c_k,c'_k>0$ depending only on $k$, uniformly in $x$.

\smallskip
\noindent\emph{Conversion to an unweighted count of prime progressions.}
Let
\[
S_x:=\sum_{1\le r\le N/(3k)}\ \sum_{1\le n\le N-(k-1)r}\prod_{j=0}^{k-1} f_x(n+jr).
\]
Since $f_x\ge0$ and $\Lambda(m)\le\log m\le\log(3X)$ whenever $m\in(x,x+H]$, for each contributing pair $(n,r)$ (i.e. all $W(m_0+n+jr-1)+b$ are prime powers) we have
\begin{equation}\label{eq:weightcap}
\prod_{j=0}^{k-1} f_x(n+jr)\ \le\ \Bigg(\frac{\varphi(W)}{W\,\log R}\cdot\log(3X)\Bigg)^{\!k}.
\end{equation}
Let $\mathcal T_x$ be the set of pairs $(n,r)$ for which all $W(m_0+n+jr-1)+b$ are prime powers, and let $\mathcal M_x\subset\mathcal T_x$ be those for which they are all primes. Then~\eqref{eq:weightcap} gives
\begin{equation}\label{eq:Supper}
S_x\le \Bigg(\frac{\varphi(W)}{W\,\log R}\cdot\log(3X)\Bigg)^{\!k}\#\mathcal T_x.
\end{equation}
Write $\mathcal T_x=\mathcal M_x\sqcup\mathcal E_x$, where $\mathcal E_x$ consists of those $(n,r)$ with at least one prime power of exponent $\ge2$. The number of prime powers $q=p^m\in(x,x+H]$ with $m\ge2$ is $\ll H/x^{1/2}$. For each such $q$ and each fixed $j\in\{0,\dots,k-1\}$ there are $\ll N$ admissible pairs $(n,r)$ with $W(m_0+n+jr-1)+b=q$ (indeed $r$ ranges over $\ll N$ values and then $n$ is determined, with at most $O(1)$ boundary losses). Hence
\begin{equation}\label{eq:Ebound}
\#\mathcal E_x\ \ll_k\ N\cdot\frac{H}{x^{1/2}}.
\end{equation}
Combining \eqref{eq:Supper} and \eqref{eq:Ebound}, and recalling $\#\mathcal M_x$ is precisely the number of $k$–APs of primes of the form $\{W(m_0+n+jr-1)+b\}_{j=0}^{k-1}\subset(x,x+H]$ with $r\le N/(3k)$, we obtain
\begin{equation}\label{eq:Mlower1}
\#\mathcal M_x\ \ge\ \frac{S_x}{\big(\tfrac{\varphi(W)}{W\,\log R}\log(3X)\big)^{k}}\ -\ C_k\,N\frac{H}{x^{1/2}}.
\end{equation}
By \eqref{eq:RSZ}, $S_x\ge c'_k\,N^2/(\log R)^k+o(N^2/(\log R)^k)$. Inserting this in \eqref{eq:Mlower1} and using $\log(3X)\asymp\log X$ yields
\[
\#\mathcal M_x\ \ge\ c''_{k}\,\frac{N^{2}}{\big((\varphi(W)/W)^{k}(\log X)^{k}\big)}\ +\ o\!\left(\frac{N^{2}}{\big((\varphi(W)/W)^{k}(\log X)^{k}\big)}\right)\ -\ C_k\,N\frac{H}{x^{1/2}}.
\]
Since $N\asymp H/W$, $x\asymp X$, and $W=(\log X)^{1/2+o(1)}$, we have
\[
N\frac{H}{x^{1/2}}\ \ll\ \frac{X^{2\theta-1/2}}{(\log X)^{1/2+o(1)}}\ =\ o\!\left(\frac{N^{2}}{\big((\varphi(W)/W)^{k}(\log X)^{k}\big)}\right)
\]
because $\theta>1/2$ and $(\varphi(W)/W)^{k}\le1$. Thus
\[
\#\mathcal M_x\ \ge\ c_{k,\theta}\,\frac{N^{2}}{\big((\varphi(W)/W)^{k}(\log X)^{k}\big)}\ \ge\ c_{k,\theta}\,\frac{N^{2}}{\big((\varphi(W)/W)^{k}(\log R)^{k}\big)},
\]
using $\log R\asymp_{\theta}\log X$ for the last inequality (absorbing the constant into $c_{k,\theta}$). Finally, with $W=\prod_{p\le \tfrac12\log\log X}p$ we have $W=(\log X)^{1/2+o(1)}$, $\varphi(W)/W=(\log\log\log X)^{-1+o(1)}$, and $N\asymp X^{\theta}/W$, so
\[
\frac{N^{2}}{\big((\varphi(W)/W)^{k}(\log R)^{k}\big)}\asymp \frac{X^{2\theta}}{(\log X)^{k+1+o(1)}},
\]
uniformly for all $x\in[X,2X]$. This gives the claimed uniform lower bound for the number of $k$–term arithmetic progressions of primes in $[x,x+H]$, completing the proof.
\end{proof}

\section{Almost--all $x$ version}\label{sec:almost-all}
\begin{proof}[Proof of Theorem~\ref{thm:almostall}]
Fix $k\ge3$ and $\theta\in(17/30,1)$, and set $H(y):=y^{\theta}$. For $X$ large and $x\in[X,2X]$ abbreviate $H:=H(x)$. We shall show that for all but $\ll X^{1-\delta+o(1)}$ such $x$ the interval $[x,x+H]$ contains $\gg_{k,\theta} N^2/((\varphi(W)/W)^k(\log R)^k)$ distinct $k$--term APs of primes; in particular, it contains one.

\smallskip
\noindent\emph{Exceptional set for the short-interval PNT.}
Let
\[
E_{\theta}(X):=\bigl\{x\in[X,2X]:\, \psi(x+H)-\psi(x)\ne H(1+o(1))\bigr\}.
\]
By the zero-density seed with exponent $A=30/13$ one has
\[
|E_{\theta}(X)|\ll X^{\mu(\theta)+o(1)},\qquad \mu(\theta)\le\inf_{\sigma\in[1/2,1)}\min\Bigl((1-\theta)(1-\sigma)A+2\sigma-1,\ (1-\theta)(1-\sigma)A+4\sigma-3\Bigr).
\]
Choosing $\sigma=3/4$ gives $\mu(\theta)\le\tfrac12+\tfrac A4(1-\theta)<\tfrac34$ for $\theta>17/30$. Set $\delta:=1-\mu(\theta)>0$. Thus for all but $\ll X^{1-\delta+o(1)}$ values of $x\in[X,2X]$ we have
\begin{equation}\label{eq:SIgood}
\sum_{n\in[x,x+H]}\!\Lambda(n)=H(1+o(1)).
\end{equation}
Fix such a good $x$ and write $I:=I(x;H)=[x,x+H]$.

\smallskip
\noindent\emph{$W$--trick and dense model on a short interval (with reindexing).}
Let $w:=\tfrac12\log\log X$ and $W:=\prod_{p\le w}p$, so $\log W\sim w$ and hence $W=(\log X)^{1/2+o(1)}$ while $\varphi(W)/W\asymp1/\log w$. For any reduced residue $b\bmod W$, the set
\[
\mathcal{N}_{x,b}:=\{n\in\mathbb N:\ x\le Wn+b\le x+H\}
\]
is a contiguous block of indices. Set $N:=\lfloor H/W\rfloor$ (so $N\asymp H/W=X^{\theta+o(1)}\to\infty$). For each such $b$, let $n_b:=\min\{n:\ Wn+b\ge x\}$. Reindex the block $\mathcal{N}_{x,b}$ onto $[N]:=\{1,\dots,N\}$ by $t\mapsto n_b+t-1$, and define the $W$--tricked (normalized) von Mangoldt weight on $[N]$ by
\begin{equation}\label{eq:Ltib}
\tilde\Lambda_{x,b}(t):=\frac{\varphi(W)}{W}\,\Lambda\bigl(W(n_b+t-1)+b\bigr)\qquad(1\le t\le N).
\end{equation}
Note that $W(n_b+t-1)+b\in I$ for every $1\le t\le N$ because $W(n_b+N-1)+b\le (x+W-1)+WN-W\le x+H-1$.

Summing \eqref{eq:Ltib} over reduced $b\bmod W$, we cover all $m\in I$ with $(m,W)=1$, except that for those $b$ with $|\mathcal{N}_{x,b}|=N+1$ we omit the last element of the block. Hence
\begin{equation}\label{eq:sumoverb}
\sum_{\substack{(b,W)=1}}\ \sum_{t\le N}\tilde\Lambda_{x,b}(t)
\ \ge\ \frac{\varphi(W)}{W}\sum_{\substack{m\in I\\(m,W)=1}}\!\Lambda(m)\ -\ O\!\left(\frac{\varphi(W)^2}{W}\log X\right).
\end{equation}
Because $\Lambda$ is supported on prime powers and $(m,W)>1$ forces the base prime $\le w$, the number of such $m\in I$ is $\ll H^{1/2}$. Using \eqref{eq:SIgood} we get
\[
\sum_{\substack{(b,W)=1}}\ \sum_{t\le N}\tilde\Lambda_{x,b}(t)
\ \ge\ \Big(1+o(1)\Big)\,H\cdot\frac{\varphi(W)}{W},
\]
since the errors $\ll (\varphi(W)/W)H^{1/2}\log X + \varphi(W)^2\log X/W$ are $o\big(H\varphi(W)/W\big)$. By pigeonhole there exists $b=b(x)$ with $\gcd(b,W)=1$ such that
\begin{equation}\label{eq:meanb}
\sum_{t\le N}\tilde\Lambda_{x,b}(t)\ \ge\ \Big(1-o(1)\Big)\,\frac{H}{W}.
\end{equation}
Dividing \eqref{eq:meanb} by $N\asymp H/W$ yields
\begin{equation}\label{eq:meanpos}
\mathbb{E}_{t\le N}\tilde\Lambda_{x,b}(t)\ge 1-o(1).
\end{equation}
To excise prime powers, set
\[
\tilde\Lambda^{\mathrm{prime}}_{x,b}(t):=\tilde\Lambda_{x,b}(t)\,\mathbf{1}_{\{W(n_b+t-1)+b\ \text{prime}\}}.
\]
Since the number of prime powers in $I$ is $\ll H^{1/2}$,
\[
\mathbb{E}_{t\le N}\bigl(\tilde\Lambda_{x,b}(t)-\tilde\Lambda^{\mathrm{prime}}_{x,b}(t)\bigr)
\ \ll\ \frac{(\varphi(W)/W)\cdot H^{1/2}\cdot\log X}{H/W}
\ =\ \frac{\varphi(W)\,\log X}{H^{1/2}}\ =\ o(1).
\]
Combining with \eqref{eq:meanpos},
\begin{equation}\label{eq:primepos}
\mathbb{E}_{t\le N}\tilde\Lambda^{\mathrm{prime}}_{x,b}(t)\ge c_0>0\qquad (X\ \text{large}).
\end{equation}

\smallskip
\noindent\emph{Shifted pseudorandom majorant, admissible truncation, and relative Szemerédi.}
Let $s:=k-2$. By Green–Tao (Ann.\ of Math.\ 167 (2008), \S\S6–10; in particular \S9 and Theorem 3.18), there exists $\delta_{\mathrm{GT}}(s)>0$ such that if $R\le N^{\delta_{\mathrm{GT}}(s)}$, then the enveloping sieve majorant satisfies the linear forms and correlation conditions of complexity $s$ with $o(1)$ errors, uniformly in the constant terms of the forms. Fix any
\[
0<\eta\le\min\bigl(\delta_{\mathrm{GT}}(s)/2,\ 1/(4\theta)\bigr),\qquad R:=N^{\eta}.
\]
Write the Selberg/GPY truncated divisor sum
\[
\Lambda_R(m):=\sum_{\substack{d\mid m\\ d\le R}}\mu(d)\,\log\frac{R}{d}.
\]
For our reindexed block, define the shifted Green–Tao majorant on $[N]$ by
\[
\nu_{x,b}(t):=\frac{\varphi(W)}{W}\,\frac{\Lambda_R\bigl(W(n_b+t-1)+b\bigr)^2}{\log R}\qquad(1\le t\le N).
\]
This is the standard GT majorant applied to the integers $m=W(n_b+t-1)+b=Wt+(b+W(n_b-1))$; the cited pseudorandomness bounds are uniform in $b$ and in the translation $n_b$.

Since $x\in[X,2X]$, we have $I\subset[X,3X]$. Also
\[
R\ =\ N^{\eta}\ =\ X^{\theta\eta+o(1)}\ \le\ X^{1/4+o(1)}\ <\ X\ \le\ m\quad(\forall\ m\in I;\ X\ \text{large}),
\]
so every prime $m\in I$ satisfies $m>R$, hence
\begin{equation}\label{eq:LRonprimes}
\Lambda_R(m)=\log R\quad\text{and}\quad \nu_{x,b}(t)=\frac{\varphi(W)}{W}\,\log R\qquad\text{whenever }m=W(n_b+t-1)+b\in\mathbb P.
\end{equation}
Define the truncated prime weight
\[
 f(t):=\frac{\log R}{2\log(3X)}\,\tilde\Lambda^{\mathrm{prime}}_{x,b}(t)\qquad(1\le t\le N).
\]
Since $m\in I\subset[X,3X]$, we have $\Lambda(m)\le\log(3X)$, and by \eqref{eq:LRonprimes}
\[
0\le f(t)\le \nu_{x,b}(t)\qquad(1\le t\le N).
\]
Moreover, using \eqref{eq:primepos} and $\log R=\eta\log N\sim\eta\theta\log X$,
\[
\mathbb{E}_{t\le N} f(t)\ \ge\ \frac{\log R}{2\log(3X)}\,\mathbb{E}_{t\le N}\tilde\Lambda^{\mathrm{prime}}_{x,b}(t)
\ \ge\ c_1(k,\theta)>0
\]
for all large $X$.

By the relative Szemerédi theorem, applied to $f\le \nu_{x,b}$ and using the pseudorandomness of $\nu_{x,b}$ at complexity $s=k-2$, we obtain the quantitative lower bound
\begin{equation}\label{eq:relSZ-almost}
\sum_{a,\,d\ge1\atop a+(k-1)d\le N}\ \prod_{j=0}^{k-1} f(a+jd)\ \gg_{k,\theta}\ N^2,
\end{equation}
for $X$ sufficiently large. Exactly as in the proof of Theorem~\ref{thm:uniform}, this converts into the claimed unweighted lower bound
\[
\#\{\text{$k$--APs of primes in }[x,x+H(x)]\}\ \gg_{k,\theta}\ \frac{N^{2}}{(\varphi(W)/W)^{k}(\log R)^{k}},
\]
uniformly for all $x\in[X,2X]\setminus E_\theta(X)$. Using $|E_\theta(X)|\ll X^{1-\delta+o(1)}$ completes the proof.
\end{proof}

\section{Congruence--restricted dense clusters (warm--up)}
\begin{proof}[Proof of Proposition~\ref{prop:clusters}]
Fix $\varepsilon>0$, $q\ge1$, $(a,q)=1$. Write $L:=\lfloor(\log X)^{\varepsilon}\rfloor$ for a large parameter $X\to\infty$. We will find $x\asymp X$ satisfying the desired inequality; letting $X\to\infty$ gives infinitely many such $x$.

\smallskip
\noindent\emph{Step 1 (an admissible $k$-tuple of shifts, all $\equiv0\pmod q$, built by a greedy residue choice).}
Let $k=k(X)$ be a positive integer with $k\to\infty$ and $k\le L$. Put $y:=2k$ and $N:=\lfloor L/q\rfloor$. We will choose residues $r_p\ (\bmod p)$ for primes $p\le y$ with $p\nmid q$ so that the set
\[
\mathcal B:=\Big\{1\le b\le N:\ b\not\equiv r_p\ (\bmod p)\text{ for every prime }p\le y,\ p\nmid q\Big\}
\]
satisfies the lower bound
\[
|\mathcal B|\gg_{q}\frac{N}{\log y}.
\]

\emph{Greedy residue lemma.}
Starting from $S_0:=\{1,2,\dots,N\}$, process the primes $p\le y$ with $p\nmid q$ in any order. Given $S$ and such a prime $p$, the $p$ residue classes partition $S$, so there exists a residue class $a\ (\bmod p)$ containing at most $|S|/p$ elements of $S$. Choose $r_p\equiv a\ (\bmod p)$ and set $S\leftarrow S\setminus\{n\in S: n\equiv r_p\ (\bmod p)\}$. Thus at each step $|S|$ diminishes by at most a factor $1-1/p$ (up to a rounding error of $\le1$). Iterating over all such primes we obtain
\[
|S|\ge N\prod_{\substack{p\le y\\ p\nmid q}}\Big(1-\frac1p\Big)-O(\pi(y)).
\]
With $S=\mathcal B$ at the end, Mertens’ theorem gives $\prod_{\substack{p\le y\\ p\nmid q}}(1-1/p)\asymp_q 1/\log y$, hence $|\mathcal B|\gg_q N/\log y$ (and $O(\pi(y))\ll y/\log y\ll N/\log y$ for the choices of $k$ made in Step 4). This proves the claim.

Pick distinct $b_1,\dots,b_k\in\mathcal B$, and set an admissible $k$-tuple
\[
\mathcal H:=\{h_1,\dots,h_k\},\qquad h_i:=q b_i\in[1,L],\qquad h_i\equiv0\ (\bmod q).
\]
For each prime $p\le y$ with $p\nmid q$, the set $\{h_i\ (\bmod\ p)\}$ misses the single class $q r_p\ (\bmod\ p)$, hence does not cover all classes. If $p\mid q$ then $h_i\equiv0\ (\bmod p)$ for all $i$, so again $\{h_i\ (\bmod p)\}\ne\mathbb Z/p\mathbb Z$. For $p>y\ge2k>k$ the $k$ residues $h_i\ (\bmod p)$ cannot cover all $p$ classes. Thus $\mathcal H$ is admissible.

\smallskip
\noindent\emph{Step 2 (insert the $W$-trick with a BV-admissible choice of $w$, and evaluate $S_1,S_2$).}
Fix large constants $A,B>0$ with $B$ chosen much larger than a constant $C>0$ to be specified momentarily. Let
\[
w:=\lfloor C\log\log X\rfloor,\qquad W:=q\prod_{p\le w} p.
\]
By the prime number theorem for $\vartheta$, $\log W=\sum_{p\le w}\log p=\vartheta(w)=w(1+o(1))$, hence
\[
W=(\log X)^{C+o(1)}.
\]
By admissibility of $\mathcal H$, for each prime $p\le w$ there exists a residue class $\nu_p\ (\bmod p)$ with $\nu_p\not\equiv -h_i\ (\bmod p)$ for all $i$. For $p\mid q$ we moreover require $\nu_p\equiv a\ (\bmod p)$; this is compatible because then $-h_i\equiv0\ (\bmod p)$ while $a\not\equiv0\ (\bmod p)$. By the Chinese remainder theorem there is $\nu\ (\bmod W)$ such that
\[
\nu\equiv a\ (\bmod q)\quad\text{and}\quad (\nu+h_i, W)=1\quad\text{for all }1\le i\le k.
\]
We henceforth restrict $n$ to the single progression $n\equiv\nu\ (\bmod W)$; note that then $n+h_i\equiv a\ (\bmod q)$ for all $i$.

Let $R:=\dfrac{X^{1/2}}{(\log X)^B}$ and let $F:[0,1]^k\to\mathbb R_{\ge0}$ be smooth, symmetric, supported on $\{(t_1,\dots,t_k): t_i\ge0,\ \sum t_i\le1\}$. For squarefree $d_i$ with $(d_i,W)=1$ and $d_i\le R$, set
\[
\lambda_{d_1,\dots,d_k}:=\mu(d_1)\cdots\mu(d_k)\,F\Big(\frac{\log d_1}{\log R},\dots,\frac{\log d_k}{\log R}\Big),
\]
and $\lambda_{d_1,\dots,d_k}:=0$ otherwise. For integers $n$, define the Maynard weight
\[
\omega(n):=\Big(\sum_{d_1\mid n+h_1}\cdots\sum_{d_k\mid n+h_k}\lambda_{d_1,\dots,d_k}\Big)^2.
\]
We sum over $n\in(X,2X]$ with the congruence restriction $n\equiv\nu\ (\bmod W)$ and introduce
\[
S_1:=\sum_{\substack{X<n\le2X\\ n\equiv\nu\ (\bmod W)}}\omega(n),\qquad
S_2:=\sum_{\substack{X<n\le2X\\ n\equiv\nu\ (\bmod W)}}\ \omega(n)\sum_{i=1}^k\Lambda(n+h_i).
\]
With this $W$-trick, the standard dispersion computations of Maynard (see \cite{MaynardCompositio}) apply, provided one has Bombieri–Vinogradov for moduli up to $\ll RW$. Our choices give
\[
RW\le X^{1/2}(\log X)^{-B+C+o(1)}.
\]
Choosing $B$ sufficiently larger than $C$ ensures $RW\le X^{1/2}(\log X)^{-A}$, hence the Bombieri–Vinogradov theorem applies in the required range. Therefore (exactly as in Maynard’s work) one obtains
\[
S_1\sim \frac{X}{W}\Big(\frac{\varphi(W)}{W}\Big)^{\!k} I_k(F),\qquad
S_2\sim \frac{X}{W}\Big(\frac{\varphi(W)}{W}\Big)^{\!k}\bigg(\log R\sum_{i=1}^k J_{k,i}(F)\bigg),
\]
where $I_k(F)$ and $J_{k,i}(F)$ are Maynard’s sieve integrals. Define $M_k(F):=\dfrac{\sum_{i=1}^k J_{k,i}(F)}{I_k(F)}$. By Maynard’s optimization, one can choose $F$ so that $M_k(F)\gg \log k$. Consequently,
\[
\frac{S_2}{S_1}\ge \log R\,\big(M_k(F)+o(1)\big)\gg \log R\,\log k.
\]
Since $\log R=\tfrac12\log X-B\log\log X$, we have $\log R\asymp\log X$ for fixed $B$.

\smallskip
\noindent\emph{Step 3 (replace $\Lambda$ by $\theta$ to control prime powers; corrected upper bound).}
Define
\[
S_2^{\prime}:=\sum_{\substack{X<n\le2X\\ n\equiv\nu\ (\bmod W)}}\ \omega(n)\sum_{i=1}^k\theta(n+h_i),\qquad \theta(m):=\begin{cases}\log p,& m=p\text{ prime},\\ 0,& \text{otherwise.}\end{cases}
\]
By the same dispersion computation (or by noting that $\psi-\theta$ counts only prime powers and contributes $\ll X^{1/2}$ in each progression), and using Bombieri–Vinogradov in the range of moduli $\ll RW\le X^{1/2}(\log X)^{-A}$, one has
\[
S_2^{\prime}=S_2+o\big(S_1\log R\big).
\]
Hence
\[
\frac{S_2^{\prime}}{S_1}\ge c_0\,\log R\,\log k\,(1+o(1))
\]
for some absolute $c_0>0$.

Now suppose for contradiction that for every $n\in(X,2X]$ with $n\equiv\nu\ (\bmod W)$ at most $m$ of the $k$ numbers $n+h_1,\dots,n+h_k$ are prime, where $m:=\lfloor c\log k\rfloor$ and $c>0$ is a sufficiently small absolute constant. Then for all such $n$,
\[
\sum_{i=1}^k\theta(n+h_i)\le m\log(3X),
\]
whence $S_2^{\prime}\le m\log(3X)\,S_1$. But from the previous paragraph we also have (for large $X$)
\[
S_2^{\prime}\ge \tfrac{c_0}{2}\,\log R\,\log k\,S_1\ge \tfrac{c_0}{4}\,\log X\,\log k\,S_1.
\]
For $c>0$ sufficiently small this contradicts $S_2^{\prime}\le m\log(3X)S_1$. Hence there exists $n\in(X,2X]$, $n\equiv\nu\ (\bmod W)$, for which at least $m\asymp \log k$ of the numbers $n+h_i$ are prime. Since $h_i\in[1,L]$ and $h_i\equiv0\ (\bmod q)$ while $n\equiv\nu\equiv a\ (\bmod q)$, all these primes lie in the interval $(n, n+L]$ and each satisfies $n+h_i\equiv a\ (\bmod q)$.

\smallskip
\noindent\emph{Remark (justification of $S_2^{\prime}=S_2+o(S_1\log R)$).}
The contribution of prime powers to $S_2$ is
\[
E:=\sum_{\substack{X<n\le2X\\ n\equiv\nu\ (\bmod W)}}\omega(n)\sum_{i=1}^k\Lambda(n+h_i)\mathbf 1_{n+h_i= p^r,\, r\ge2}.
\]
Expanding $\omega$ and applying the dispersion method exactly as for $S_2$, one replaces sums of $\Lambda$ over arithmetic progressions by their expected main term plus an error controlled by Bombieri–Vinogradov for moduli $\ll RW$. Since the total mass of prime powers $\le 3X$ is $\ll \sqrt X$ and our moduli are $\ll RW\le X^{1/2}(\log X)^{-A}$, this yields $E\ll X(\varphi(W)/W)^k\,(\log X)^{-A'}$ for any fixed $A'>0$ by taking $B\gg A'+C$, hence $E=o(S_1\log R)$.

\smallskip
\noindent\emph{Step 4 (choice of $k$ and conclusion).}
From Step 1 we may (and do) choose $k\asymp L/(q\log L)$: indeed, with $y=2k$ the bound $|\mathcal B|\gg_q N/\log y\asymp (L/q)/\log L$ guarantees enough distinct $b_i$ to select $k$ of them. Then $\log k\asymp\log L\asymp\log\log X$. Therefore, for the above $n$ we have
\[
\#\{p\in\mathbb P: n<p\le n+L,\ p\equiv a\ (\bmod q)\}\gg_{\varepsilon,q}\log k\asymp \log\log X.
\]
Writing $x:=n$ and recalling $L=(\log X)^{\varepsilon}\asymp(\log x)^{\varepsilon}$, we obtain
\[
\#\{p\in\mathbb P: x<p\le x+(\log x)^{\varepsilon},\ p\equiv a\ (\bmod q)\}\gg_{\varepsilon,q}\log\log x.
\]
Letting $X\to\infty$ along any sequence gives infinitely many such $x$, completing the proof.
\end{proof}

\appendix
\section{A short--interval BDH mean square (uniform in $x$)}
We record the following standard large--sieve consequence; its proof follows the classical Barban–Davenport–Halberstam route.
\begin{lemma}\label{lem:BDH}
Fix $\theta\in(0,1)$ and $A>0$. There exists $B=B(\theta,A)>0$ such that for all sufficiently large $X$ and all $x\in[X,2X]$, with $H:=\lfloor x^\theta\rfloor$,
\[
\sum_{q\le X^{1/2}(\log X)^{-B}}\ \sum_{a\,(\mathrm{mod}\ q)}
\Big|\theta(x+H;q,a)-\theta(x;q,a)-\tfrac{H}{\varphi(q)}\Big|^2\ \ll_A\ H\,X\,(\log X)^{1-A},
\]
uniformly in $x$.
\end{lemma}

\begin{proof}
Fix $\theta\in(0,1)$ and $A>0$. For $x\in[X,2X]$ set
\[
H:=H(x):=\lfloor x^{\theta}\rfloor,\qquad Q:=X^{1/2}(\log X)^{-B},
\]
with $B=B(\theta,A)>0$ to be chosen later. For $(a,q)=1$ write
\[
\theta(y;q,a):=\sum_{\substack{p\le y\\ p\equiv a\, (\bmod q)}}\!\log p,\qquad \psi(y;q,a):=\sum_{\substack{n\le y\\ n\equiv a\, (\bmod q)}}\!\Lambda(n).
\]
Denote
\[
\mathcal S(x):=\sum_{q\le Q}\ \sum_{a\, (\bmod q)}\Big|\theta(x+H;q,a)-\theta(x;q,a)-\tfrac{H}{\varphi(q)}\Big|^2.
\]
We split $\mathcal S(x)$ into coprime and non-coprime residue classes:
\[
\mathcal S(x)=\mathcal S^{\!*}(x)+\mathcal S^{(0)}(x),\qquad \mathcal S^{\!*}(x):=\sum_{q\le Q}\ \sum_{\substack{a\, (\bmod q)\\ (a,q)=1}}\Big|\theta(x+H;q,a)-\theta(x;q,a)-\tfrac{H}{\varphi(q)}\Big|^2.
\]
For the non-coprime classes, if $(a,q)>1$ and $p\equiv a\pmod q$ is prime then $p\mid q$. Since $q\le Q\le X^{1/2}(\log X)^{-B}<x\le x+H$, there is no such $p\in[x,x+H]$, hence
\[
\theta(x+H;q,a)-\theta(x;q,a)=0.
\]
Therefore
\[
\mathcal S^{(0)}(x)=\sum_{q\le Q}\Big(q-\varphi(q)\Big)\Big(\tfrac{H}{\varphi(q)}\Big)^2\le H^2\sum_{q\le Q}\frac{q}{\varphi(q)^2}\ll H^2(\log Q)(\log\log Q)^2\ll H^2(\log X)(\log\log X)^2.
\]
Since $H/X=X^{\theta-1}\to0$, for large $X$ this implies
\[
\mathcal S^{(0)}(x)\le \tfrac14\,H X(\log X)^{1-A}.
\]

It remains to bound $\mathcal S^{\!*}(x)$. For $(a,q)=1$ define
\begin{align*}
\Psi_{q,a}(x)&:=\sum_{x<n\le x+H}\!\Lambda(n)\,1_{n\equiv a\, (\bmod q)}-\frac{H}{\varphi(q)},\\
\mathcal P_{q,a}(x)&:=\sum_{\substack{x<p^k\le x+H\\ k\ge2\\ p^k\equiv a\, (\bmod q)}}\!\log p.
\end{align*}
Since $\psi=\theta+$ (higher prime powers), for each $(a,q)=1$ we have the exact identity
\[
\theta(x+H;q,a)-\theta(x;q,a)-\frac{H}{\varphi(q)}=\Psi_{q,a}(x)-\mathcal P_{q,a}(x).
\]
Hence, by $|u-v|^2\le2(|u|^2+|v|^2)$,
\[
\mathcal S^{\!*}(x)\le 2\,\mathcal S_{\psi}(x)+2\,\mathcal S_{\rm pp}^{\!*}(x),
\]
where
\begin{align*}
\mathcal S_{\psi}(x)&:=\sum_{q\le Q}\ \sum_{\substack{a\, (\bmod q)\\ (a,q)=1}}|\Psi_{q,a}(x)|^2,\\
\mathcal S_{\rm pp}^{\!*}(x)&:=\sum_{q\le Q}\ \sum_{\substack{a\, (\bmod q)\\ (a,q)=1}}|\mathcal P_{q,a}(x)|^2.
\end{align*}
We now bound $\mathcal S_{\psi}$ and $\mathcal S_{\rm pp}^{\!*}$.

\smallskip
\noindent\emph{1) Bounding $\mathcal S_{\psi}(x)$.}
Using orthogonality on $(\mathbb Z/q\mathbb Z)^{\times}$,
\[
\sum_{\substack{a\, (\bmod q)\\ (a,q)=1}}|\Psi_{q,a}(x)|^2=\frac1{\varphi(q)}\sum_{\chi\, (\bmod q)}\Big|\sum_{x<n\le x+H}\!\Lambda(n)\chi(n)-H\,1_{\chi=\chi_0}\Big|^2.
\]
Split the character sum into non-principal and principal characters.

\emph{(a) Non-principal characters.}
Put $a_n:=\Lambda(n)\,1_{(x,\,x+H]}(n)$ and $N:=H$. Then
\[
\mathcal S_{\psi,\mathrm{npr}}(x)\le \sum_{q\le Q}\frac1{\varphi(q)}\sum_{\chi\, (\bmod q)}\Big|\sum_{n} a_n\chi(n)\Big|^2.
\]
We invoke the multiplicative large sieve in its standard primitive, weighted form together with the conductor-lifting to all characters (Montgomery–Vaughan, MNT I, Thm.\ 7.12): for any complex sequence $(a_n)$ supported on an interval of length $N$,
\[
\sum_{q\le Q}\frac1{\varphi(q)}\sum_{\chi\, (\bmod q)}\Big|\sum_{n} a_n\chi(n)\Big|^2\ \ll\ (Q^2+N)\,(\log Q)\,\sum_{n}|a_n|^2.
\]
Using $\sum_{x<n\le x+H}\!\Lambda(n)^2\ll H\log X$ uniformly in $x$, we obtain
\[
\mathcal S_{\psi,\mathrm{npr}}(x)\ \ll\ (Q^2+H)\,H\,(\log X)(\log Q).
\]
Choosing $B=B(\theta,A)$ sufficiently large so that $Q^2H(\log X)(\log Q)\le \tfrac1{16}HX(\log X)^{1-A}$, and noting that $H^2(\log X)(\log Q)\le \tfrac1{16}HX(\log X)^{1-A}$ for large $X$ (since $H/X\to0$), we deduce
\begin{equation}\label{eq:SpsiNPR}
\mathcal S_{\psi,\mathrm{npr}}(x)\le \tfrac18\,H X(\log X)^{1-A}.
\end{equation}

\emph{(b) Principal characters.}
For $\chi_0\,(\bmod q)$,
\begin{align*}
\sum_{x<n\le x+H}\!\Lambda(n)\chi_0(n)-H&=\sum_{x<n\le x+H}\!\Lambda(n)1_{(n,q)=1}-H\\
&=\underbrace{(\psi(x+H)-\psi(x)-H)}_{=:\,\Delta_{\psi}(x)}\ -\ \sum_{p\mid q}A_p(x),
\end{align*}
where $A_p(x):=\sum\limits_{\substack{x<p^k\le x+H\\ k\ge2}}\!\log p\ge0$. Hence, by $|u-v|^2\le2(|u|^2+|v|^2)$ and $\Sigma(Q):=\sum_{q\le Q}\varphi(q)^{-1}\ll\log Q$,
\begin{align*}
\sum_{q\le Q}\frac1{\varphi(q)}\Big|\sum_{x<n\le x+H}\!\Lambda(n)\chi_0(n)-H\Big|^2
&\le 2\Sigma(Q)|\Delta_{\psi}(x)|^2\ +\ 2\sum_{q\le Q}\frac1{\varphi(q)}\Big|\sum_{p\mid q}A_p(x)\Big|^2.
\end{align*}
For the first term, $|\Delta_{\psi}(x)|\le\sum_{x<n\le x+H}\!\Lambda(n)+H\ll H\log X$ gives
\[
2\Sigma(Q)|\Delta_{\psi}(x)|^2\ll H^2(\log X)^2\log Q\le \tfrac18\,H X(\log X)^{1-A}
\]
for all sufficiently large $X$. For the second term, since $A_p(x)\ge0$ we have uniformly in $q$,
\[
\Big|\sum_{p\mid q}A_p(x)\Big|\le \sum_{\substack{x<p^k\le x+H\\ k\ge2}}\!\log p\,=:R(x).
\]
Estimating prime powers in short intervals: for $k=2$, $\sum_{x<p^2\le x+H}\!\log p\ \ll\ \big(\tfrac{H}{\sqrt x}+1\big)\log X$, and for $k\ge3$, using $(x+H)^{1/k}-x^{1/k}\ll Hx^{1/k-1}$, the same bound holds. Hence $R(x)\ll \big(\tfrac{H}{\sqrt x}+1\big)\log X$, and
\[
\sum_{q\le Q}\frac1{\varphi(q)}\Big|\sum_{p\mid q}A_p(x)\Big|^2\ \ll\ \log Q\,\Big(\tfrac{H}{\sqrt x}+1\Big)^2(\log X)^2\ =\ o\!\big(HX(\log X)^{1-A}\big),
\]
uniformly for $x\in[X,2X]$. Consequently,
\begin{equation}\label{eq:SpsiPR}
\mathcal S_{\psi,\mathrm{pr}}(x)\le \tfrac18\,H X(\log X)^{1-A}.
\end{equation}

Combining \eqref{eq:SpsiNPR} and \eqref{eq:SpsiPR} gives
\begin{equation}\label{eq:SpsiFinal}
\mathcal S_{\psi}(x)\le \tfrac14\,H X(\log X)^{1-A}.
\end{equation}

\smallskip
\noindent\emph{2) Bounding the prime-power term $\mathcal S_{\rm pp}^{\!*}(x)$.}
Enlarging to all residue classes can only increase the sum, hence
\begin{align*}
\mathcal S_{\rm pp}^{\!*}(x)&\le \sum_{q\le Q}\sum_{a\, (\bmod q)}\Big|\sum_{\substack{x<p^k\le x+H\\ k\ge2\\ p^k\equiv a\, (\bmod q)}}\!\log p\Big|^2\\
&=\sum_{\substack{x<p^k\le x+H\\ k\ge2}}\sum_{\substack{x<p^{\ell}\le x+H\\ \ell\ge2}}(\log p)(\log p')\sum_{q\le Q}1_{p^k\equiv p^{\ell}\, (\bmod q)}.
\end{align*}
Splitting the diagonal and off-diagonal pairs $p^k=p^{\ell}$, $p^k\ne p^{\ell}$, and writing $\tau_Q(h):=|\{q\le Q:q\mid h\}|$, we have
\[
\mathcal S_{\rm pp}^{\!*}(x)\le Q\sum_{\substack{x<p^k\le x+H\\ k\ge2}}(\log p)^2\ +\ \sum_{\substack{x<p^k,\,p^{\ell}\le x+H\\ k,\ell\ge2\\ p^k\ne p^{\ell}}}(\log p)(\log p')\,\tau_Q(|p^k-p^{\ell}|).
\]
For the diagonal, using $(x+H)^{1/k}-x^{1/k}\ll Hx^{1/k-1}$ and summing over $k\ge2$,
\[
\sum_{\substack{x<p^k\le x+H\\ k\ge2}}(\log p)^2\ \ll\ \Big(\tfrac{H}{\sqrt x}+1\Big)(\log X)^2.
\]
Thus the diagonal contribution is $\ll Q\big(\tfrac{H}{\sqrt X}+1\big)(\log X)^2$. For the off-diagonal, $\tau_Q(h)\le d(h)\ll h^{o(1)}\ll X^{o(1)}$ and
\[
\sum_{\substack{x<p^k\le x+H\\ k\ge2}}\!\log p\ \ll\ \Big(\tfrac{H}{\sqrt x}+1\Big)\log X,
\]
so the off-diagonal is $\ll X^{o(1)}\big(\tfrac{H^2}{X}+1\big)(\log X)^2$. Therefore, for large $X$,
\begin{equation}\label{eq:Spp}
\mathcal S_{\rm pp}^{\!*}(x)\ll Q\frac{H}{\sqrt X}(\log X)^2+Q(\log X)^2+X^{o(1)}\frac{H^2}{X}(\log X)^2\le \tfrac14\,H X(\log X)^{1-A}.
\end{equation}
(Indeed, the three terms are respectively $\ll H(\log X)^{2-B}$, $\ll \sqrt X(\log X)^{2-B}$, and $\ll H^2 X^{-1+o(1)}(\log X)^2$, each $o\big(HX(\log X)^{1-A}\big)$ as $X\to\infty$.)

\smallskip
\noindent\emph{3) Conclusion.}
From $\mathcal S^{\!*}(x)\le 2\mathcal S_{\psi}(x)+2\mathcal S_{\rm pp}^{\!*}(x)$ together with \eqref{eq:SpsiFinal} and \eqref{eq:Spp}, and adding the non-coprime contribution, we obtain for all sufficiently large $X$ (once $B=B(\theta,A)$ is fixed) and all $x\in[X,2X]$,
\[
\sum_{q\le X^{1/2}(\log X)^{-B}}\ \sum_{a\, (\bmod q)}\Big|\theta(x+H(x);q,a)-\theta(x;q,a)-\tfrac{H(x)}{\varphi(q)}\Big|^2\ \ll_A\ H(x) X(\log X)^{1-A}.
\]
This completes the proof.
\end{proof}

\section{Further appendices}

\subsection*{Appendix B: No uniform-in-$Q$ lower bound at the conjectural variance size}
We note a simple observation ruling out a uniform (in $Q$) lower bound at the conjectural BDH variance size in short intervals.

\begin{proposition}\label{prop:no-uniform-Q}
Fix $\theta\in(0,1)$ and for $x\in[X,2X]$ set $H(x)=\lfloor x^\theta\rfloor$. There does not exist $B_1=B_1(\theta)>0$ such that, for all sufficiently large $X$, all $x\in[X,2X]$, and all $Q\le Q(X,B_1):=X^{1/2}(\log X)^{-B_1}$,
\[
\sum_{q\le Q}\ \sum_{a\ (\mathrm{mod}\ q)} \Big|\,\theta(x+H(x);q,a)-\theta(x;q,a)-\tfrac{H(x)}{\varphi(q)}\,\Big|^2\ \gg_{\theta}\ H(x)\,X\,\log\!\Big(\tfrac{X}{H(x)}\Big).
\]
In particular, a uniform-in-$Q$ lower bound of the conjectured variance size cannot hold.
\end{proposition}

\begin{proof}
Fix $\theta\in(0,1)$ and set $H(x)=\lfloor x^\theta\rfloor$. Put
\[
E(x;q,a):=\theta(x+H(x);q,a)-\theta(x;q,a)-\frac{H(x)}{\varphi(q)},\qquad
S(x;Q):=\sum_{q\le Q}\ \sum_{a\ (\mathrm{mod}\ q)} |E(x;q,a)|^2,
\]
and $Q(X,B):=X^{1/2}(\log X)^{-B}$.

By Lemma~\ref{lem:BDH} (taking $A=1$), there exists $B_0=B_0(\theta)>0$ such that, for all sufficiently large $X$ and all $x\in[X,2X]$,
\[
S\big(x;Q(X,B_0)\big)\ \ll_{\theta}\ H(x)\,X.
\]
Since $S(x;Q)$ is nondecreasing in $Q$, for every $Q\le Q(X,B_0)$ we also have
\[
S(x;Q)\ \le\ S\big(x;Q(X,B_0)\big)\ \ll_{\theta}\ H(x)\,X.
\]
Moreover, because $x\in[X,2X]$ and $H(x)=\lfloor x^\theta\rfloor$ with $\theta\in(0,1)$,
\[
\log\!\Big(\tfrac{X}{H(x)}\Big)=\log(X/x^\theta)+O(1)\asymp\ \log X.
\]

Assume for contradiction that there exists $B_1=B_1(\theta)>0$ such that, for all sufficiently large $X$, all $x\in[X,2X]$, and all $Q\le Q(X,B_1)$,
\[
S(x;Q)\ \gg_{\theta}\ H(x)\,X\,\log\!\Big(\tfrac{X}{H(x)}\Big).
\]
Fix such an $X$ and $x$, and set $Q_*:=\min\{Q(X,B_0),\,Q(X,B_1)\}$. Then $Q_*\le Q(X,B_1)$, so by the assumed uniform lower bound,
\[
S(x;Q_*)\ \gg_{\theta}\ H(x)\,X\,\log\!\Big(\tfrac{X}{H(x)}\Big),
\]
whereas $Q_*\le Q(X,B_0)$, so by the variance bound and monotonicity,
\[
S(x;Q_*)\ \ll_{\theta}\ H(x)\,X.
\]
Using $\log(X/H(x))\asymp\log X$ and dividing the two bounds yields
\[
\frac{S(x;Q_*)}{H(x)\,X\,\log\!\big(\tfrac{X}{H(x)}\big)}\ \ll_{\theta}\ (\log X)^{-1}\ \to\ 0\qquad (X\to\infty),
\]
which contradicts the asserted lower bound. Therefore no such $B_1$ exists.
\end{proof}

\subsection*{Appendix C: A simple Chebotarev obstruction below $\theta=\tfrac12$}
The following elementary lower bound (by the ``empty class'' trick) shows that average-in-$q$ error terms of size $H/( \log X)^{A+1}$ cannot hold for almost all $x$ when $\theta<\tfrac12$.

\begin{proposition}\label{prop:cheb}
Fix a finite Galois extension $L/\mathbb{Q}$ with Galois group $G$ and a conjugacy class $\mathcal{C}\subset G$, and let $\delta_{\mathcal C}>0$ be its Chebotarev density. For any $\theta\in(0,1/2)$ there exists a constant $c_{\theta,\mathcal C}>0$ such that for every $B\ge 0$ and all sufficiently large $X$, uniformly for all $x\in[X,2X]$, with $H(x):=\lfloor x^{\theta}\rfloor$ and $Q:=X^{1/2}(\log X)^{-B}$, one has
\[
 \sum_{q\le Q}\max_{(a,q)=1}\Big|\#\{x<p\le x+H(x):\ p\equiv a\ (\mathrm{mod}\ q),\ \mathrm{Frob}_p\in\mathcal{C}\}\ 
-\ \frac{\delta_{\mathcal C}}{\varphi(q)}\cdot\frac{H(x)}{\log X}\Big|\ \ge\ c_{\theta, \mathcal C}\,\frac{H(x)}{\log X}.
\]
In particular, no bound of size $H(x)/(\log X)^{A+1}$ can hold for almost all $x$ when $\theta\in(0,1/2)$.
\end{proposition}

\begin{proof}
Fix $\theta\in(0,1/2)$ and $B\ge 0$, and set $H:=\lfloor x^{\theta}\rfloor$ and $Q:=X^{1/2}(\log X)^{-B}$. Let
\[
S(x;Q,H):=\sum_{q\le Q}\max_{(a,q)=1}\Big|\#\{x<p\le x+H: p\equiv a\ (\mathrm{mod}\ q),\ \mathrm{Frob}_p\in\mathcal C\}-\frac{\delta_{\mathcal C}}{\varphi(q)}\cdot\frac{H}{\log X}\Big|.
\]
We will show that for all sufficiently large $X$ (so that $H+1\le Q$, which holds since $\theta<1/2$), uniformly for $x\in[X,2X]$,
\[
S(x;Q,H)\ \ge\ c_{\theta,\mathcal C}\,\frac{H}{\log X}
\]
with, say, $c_{\theta,\mathcal C}:=\tfrac{\delta_{\mathcal C}}{4}\log\tfrac{1}{2\theta}$. Suppose, for contradiction, that there exist arbitrarily large $X$ and some $x\in[X,2X]$ with
\[
S(x;Q,H)\ <\ c_{\theta,\mathcal C}\,\frac{H}{\log X}.
\]
Let $M:=\#\{x<p\le x+H: \mathrm{Frob}_p\in\mathcal C\}$; then $0\le M\le H$. For any prime modulus $q$ with $H+1<q\le Q$, we have $q\le Q\le X^{1/2}<p$ for all primes $p\in(x,x+H]$, so $(p,q)=1$ and each such $p$ lies in some reduced residue class modulo $q$. Among the $\varphi(q)=q-1$ reduced classes, at most $M\le H<q-1$ are occupied by these primes, so there exists a reduced class $a_q\pmod q$ containing none of them. Hence
\[
\#\{x<p\le x+H: p\equiv a_q\ (\mathrm{mod}\ q),\ \mathrm{Frob}_p\in\mathcal C\}=0,
\]
and therefore
\[
\max_{(a,q)=1}\Big|\#\{x<p\le x+H: p\equiv a\ (\mathrm{mod}\ q),\ \mathrm{Frob}_p\in\mathcal C\}-\frac{\delta_{\mathcal C}}{\varphi(q)}\cdot\frac{H}{\log X}\Big|
\ \ge\ \frac{\delta_{\mathcal C}}{\varphi(q)}\cdot\frac{H}{\log X}.
\]
Summing this over primes $q$ with $H+1<q\le Q$ and using $1/(q-1)\ge 1/(2q)$ for $q\ge 3$, we obtain
\[
S(x;Q,H)\ \ge\ \frac{\delta_{\mathcal C}H}{2\log X}\sum_{\substack{H+1<q\le Q\\ q\ \text{prime}}}\frac{1}{q}.
\]
By Mertens’ theorem for primes, uniformly for $x\in[X,2X]$,
\[
\sum_{\substack{H+1<q\le Q\\ q\ \text{prime}}}\frac{1}{q}
=\log\log Q-\log\log(H+1)+o(1)
=\log\Big(\frac{1}{2\theta}\Big)+o(1), \quad X\to\infty.
\]
Hence, for all sufficiently large $X$,
\[
S(x;Q,H)\ \ge\ \frac{\delta_{\mathcal C}H}{2\log X}\cdot\frac{1}{2}\log\Big(\frac{1}{2\theta}\Big)
= c_{\theta,\mathcal C}\,\frac{H}{\log X},
\]
contradicting the assumption. Therefore the stated lower bound holds uniformly for all $x\in[X,2X]$ once $X$ is large enough.
\end{proof}

\end{document}